\documentclass{article}

 \usepackage{amsmath,amssymb,amsthm,mathrsfs, tikz, graphicx}
 \usepackage{setspace, varioref}
 \usepackage{enumitem}
 \usepackage{hyperref}
 \usepackage{array}
\newtheorem{theo}{Theorem}
\labelformat{theo}{Theorem~#1}
\newtheorem{lemme}{Lemma}
\labelformat{lemme}{Lemma~#1}
\newtheorem{defi}{Definition}
\labelformat{defi}{definition~#1}
\newtheorem{hyp}{Assumption}
\labelformat{hyp}{Assumption~#1}
\newtheorem{coro}{Corollary}
\labelformat{coro}{Corollary~#1}
\labelformat{section}{Section~#1}
\labelformat{subsection}{Subsection~#1}
\labelformat{figure}{Figure~#1}
\numberwithin{lemme}{section}
\numberwithin{theo}{section}
\numberwithin{defi}{section}
\numberwithin{coro}{section}

\newcommand{\RR}{\mathbb{R}}
\newcommand{\NN}{\mathbb{N}}

\newcommand{\Ss}{\mathbb{S}^{d-1}}
\newcommand{\vp}{\varphi}
\newcommand{\vr}{\varrho}
\newcommand{\ep}{\varepsilon}
\newcommand{\Ep}{^{\ep}}

\newcommand{\Oep}{\Omega_{\varepsilon}}
\newcommand{\Eep}{E^{\varepsilon}}
\newcommand{\Nep}{N^{\varepsilon}}

\newcommand{\Ob}{\overline{\Omega}}

\newcommand{\Mc}{\mathcal{M}}

\newcommand{\Pc}{\mathcal{P}}
\newcommand{\Lc}{\mathcal{L}}
\newcommand{\Gc}{\mathcal{G}}
\newcommand{\Fc}{\mathcal{F}}
\newcommand{\Hc}{\mathcal{H}}
\newcommand{\Qc}{\mathcal{Q}}
\newcommand{\Gf}{\mathbf{G}}
\newcommand{\Ff}{\mathbf{F}}

\newcommand{\OS}{\Omega \times \Ss}
\newcommand{\ObS}{\Ob \times \Ss}
\newcommand{\IO}{\int_{\Omega}}
\newcommand{\IOb}{\int_{\Ob}}

\newcommand{\IOS}{\int_{\Omega \times \Ss}}

\newcommand{\bal}{\begin{align*}}
\newcommand{\eal}{\end{align*}}
\newcommand{\ble}{\begin{lemme}}
\newcommand{\ele}{\end{lemme}}
\newcommand{\bpr}{\begin{proof}}
\newcommand{\epr}{\end{proof}}

\newcommand{\beq}{\begin{equation}}
\newcommand{\eeq}{\end{equation}}
\newcommand{\loc}{\text{loc}}

\begin{document}

\author{Rom\'eo Hatchi \thanks{\scriptsize CEREMADE, UMR CNRS 7534, Universit\'e Paris IX
Dauphine, Pl. de Lattre de Tassigny, 75775 Paris Cedex 16, FRANCE
\texttt{hatchi@ceremade.dauphine.fr}.}}

\title{Wardrop equilibria : long-term variant, degenerate anisotropic PDEs and numerical approximations}

\maketitle

  \abstract{As shown in \cite{Romeotheboss}, under some structural assumptions, working on congested traffic problems in general and increasingly dense networks leads, at the limit by $\Gamma$-convergence, to continuous minimization problems posed on measures on generalized curves. Here, we show the equivalence with another problem that is the variational formulation of an anisotropic, degenerate and elliptic PDE. For particular cases, we prove a Sobolev regularity result for the minimizers of the minimization problem despite the strong degeneracy and anisotropy of the Euler-Lagrange equation of the dual. We extend the analysis of \cite{brasco2013congested} to the general case. Finally, we use the method presented in \cite{benamou2013augmented} to make numerical simulations. }


\section{Introduction} \label{intro}

Researchers in the field of modeling traffic have developed the concept of congestion in networks since the early 50's and the introduction of the notion of Wardrop equilibrium (see \cite{wardrop1952road}). Its important popularity is due to some applications to road traffic and communication networks. We will describe the general congested network model built in \cite{Romeotheboss} in the following subsection.

\subsection{Presentation of the general discrete model}

Given $d \in \NN, d \geq 2$ and $\Omega$ a bounded domain of $\RR^d$ with a Lipschitz boundary and $\ep > 0$, we take a sequence of finite oriented networks $\Oep = (N^{\ep}, E^{\ep})$ whose characteristic length is $\ep$, where $N^{\ep}$ is the set of nodes in $\Omega_{\ep}$ and $E^{\ep}$ the set of pairs $(x,e)$ with $x \in \Nep$ and $e \in \RR^d$ such that the segment $[x, x+e]$ is included in $\Omega$. We will simply identify arcs to pairs $(x,e)$. We assume $|E^{\ep}| = \max \{|e|, \text{ there exists }x \text{ such that }(x,e) \in E^{\ep} \} = \ep$. .

\textbf{Masses and congestion:} Let us denote the traffic flow on the arc $(x,e)$ by $m^{\ep}(x,e)$. There is a function $g\Ep : \Eep \times \RR_+ \rightarrow \RR_+$ such that for each $(x,e) \in \Eep$ and $m \geq 0$, $g\Ep(x,e,m)$ represents the traveling time of arc $(x,e)$ when the mass on $(x,e)$ is $m$. The function $g\Ep$ is positive and increasing in its last variable. This describes the congestion effect. We will denote the collection of all arc-masses $m\Ep(x,e)$ by $\mathbf{m\Ep}$.

\textbf{Marginals:} There is a distribution of sources $f_-\Ep= \sum_{x \in N\Ep} f_-\Ep(x) \delta_x$ and sinks $f_+\Ep= \sum_{x \in N\Ep} f_+\Ep(x) \delta_x$ which are discrete measures with same total mass on the set of nodes $\Nep$ (that we can assume to be 1 as a normalization)
\[
\sum_{x \in N\Ep} f_-\Ep(x) = \sum_{y \in N\Ep} f_+\Ep(y) = 1.
\]
The numbers $f_-\Ep(x)$ and $f_+\Ep(x)$ are nonnegative for every $x \in \Nep$.

\textbf{Paths and equilibria:} A path is a finite set of successive arcs $(x,e) \in \Eep$ on the network. $C\Ep$ is the finite set of loop-free paths on $\Oep$ and may be partitioned as
\[
C\Ep = \bigcup_{(x,y) \in N\Ep \times N\Ep} C_{x,y}\Ep = \bigcup_{x \in N\Ep} C_{x,\cdot}\Ep = \bigcup_{y \in N\Ep} C_{\cdot,y}\Ep,
\]
where $C\Ep_{x,\cdot}$ (respectively $C\Ep_{\cdot,y}$) is the set of loop-free paths starting at the origin $x$ (respectively stopping at the terminal point $y$) and $C\Ep_{x,y}$ is the intersection of $C\Ep_{x, \cdot}$ and $C\Ep_{\cdot, y}$. Then the travel time of a path $\gamma \in C\Ep$ is given by:
\[
\tau_{\mathbf{m^{\ep}}}^{\ep}(\gamma)= \sum_{(x,e) \subset \gamma} g^{\ep} (x,e,m\Ep(x,e)).
\]

The mass commuting on the path $\gamma \in C\Ep$ will be denoted $w\Ep(\gamma)$. The collection of all path-masses $w\Ep(\gamma)$ will be denoted $\mathbf{w\Ep}$. We may define an equilibrium that satisfies optimality requirements compatible with the distribution of sources and sinks and such that all paths used minimize the traveling time between their extremities, taking into account the congestion effects. In other words, we have to impose mass conservation conditions that relate arc-masses, path-masses and the data $f_-\Ep$ and $f_+\Ep$:
\begin{equation} \label{cons1}
f_-\Ep (x) := \sum_{\gamma \in C_{x, \cdot}\Ep} w\Ep(\gamma), \: f_+\Ep (y) := \sum_{\gamma \in C_{\cdot,y}\Ep} w\Ep(\gamma), \: \forall (x, y) \in \Nep \times \Nep
\end{equation}
and
\begin{equation} \label{cons2}
m\Ep(x, e) = \sum_{\gamma \in C\Ep: (x,e) \subset \gamma} w\Ep(\gamma), \forall (x,e) \in \Eep.
\end{equation}

We define $T\Ep_{g\Ep}$ to be the minimal length functional, that is:
\[
T\Ep_{g\Ep} (x,y) := \min_{\gamma \in C\Ep_{x,y}} \sum_{(x,e) \subset \gamma} g\Ep(x,e,m\Ep(x,e).
\]
Let $\Pi(f_-\Ep, f_+\Ep)$ be the set of discrete transport plans between $f_-\Ep$ and $f_+\Ep$, that is, the set of collection of nonnegative elements $(\vp\Ep(x,y))_{(x,y) \in {N\Ep}^2}$ such that
\[
\sum_{y \in N\Ep} \vp\Ep(x,y) = f_-\Ep(x) \text{ and } \sum_{x \in N\Ep} \vp\Ep(x,y) = f_+\Ep(x) \text{, for every } (x,y) \in N\Ep \times N\Ep.
\]
 This results in the concept of Wardrop equilibrium that is defined precisely as follows:

\begin{defi} \label{defW}
A Wardrop equilibrium is a configuration of nonnegative arc-masses $\mathbf{m}\Ep : (x, e) \rightarrow (m\Ep(x,e))$ and of nonnegative path-masses $\mathbf{w}\Ep : \gamma \rightarrow w\Ep(\gamma)$, that satisfy the mass conservation conditions \eqref{cons1} and \eqref{cons2} and such that:
\begin{enumerate}
\item
For every $(x, y) \in  \Nep \times \Nep$ and every $\gamma \in C_{x,y}\Ep $, if $w\Ep(\gamma) > 0$ then
\beq \label{eqW1}
\tau_{\mathbf{m^{\ep}}}^{\ep}(\gamma) = \min_{\gamma' \in C_{x,y}\Ep} \tau_{\mathbf{m^{\ep}}}^{\ep}(\gamma'),
\eeq
\item
if we define $\Pi\Ep(x,y) = \sum_{\gamma \in C\Ep_{x,y}} w\Ep(\gamma)$ then $\Pi\Ep$ is a minimizer of
\beq \label{eqW2}
\inf_{\vp\Ep \in \Pi(f_-\Ep, f_+\Ep)} \vp\Ep(x,y) T\Ep_{g\Ep} (x,y).
\eeq
\end{enumerate}
\end{defi}

Condition \eqref{eqW1} means that users behave rationally and always use shortest paths, taking in consideration congestion, that is, travel times increase with the flow. In \cite{baillon2012discrete, Romeotheboss}, the main discrete model studied is short-term, that is, the transport plan is prescribed. Here we work with a long-term variant as in \cite{brasco2013congested, brasco2010congested}. It means that we have fixed only the marginals (that are $f_-\Ep$ and $f_+\Ep$). So the transport plan now is an unknown and must be determined by some additional optimality condition that is \eqref{eqW2}. Condition \eqref{eqW2} requires that there is an optimal transport plan between the fixed marginals for the transport cost induced by the congested metric. So we also have an optimal transportation problem.

\subsection{Assumptions and preliminary results}

A few years after the work of Wardrop, Beckmann, McGuire and Winsten \cite{beckmann1956studies} observed that Wardrop equilibria coincide with the minimizers of a convex optimization problem:
\begin{theo} A flow configuration $(\mathbf{w}^{\ep}, \mathbf{m}^{\ep})$ is a Wardrop equilibrium if and only if it minimizes
\begin{equation} \label{P1}
\sum_{(x,e) \in E^{\ep}} G^{\ep} (x, e, m^{\ep}(x,e)) \text{ where } G^{\ep}(x, e, m) := \int_0^m g^{\ep}(x, e, \alpha) d \alpha
\end{equation}
subject to nonnegativity constraints and the mass conservation conditions \eqref{cons1}-\eqref{cons2}.
\end{theo}

The problem \eqref{P1} is interesting since it easily implies existence results and numerical schemes. However, it requires knowing the whole path flow configuration $\mathbf{w\Ep}$ so that it may quickly be untractable for dense networks. However a similar issue was recently studied in \cite{Romeotheboss}. Under structural assumptions, it is shown that we may pass to a continuous limit which will simplify the structure. Here, we will not see all these hypothesis, only the main ones. So we refer to \cite{Romeotheboss} for more details.

 \begin{hyp} \label{hy1}
 The discrete measures $(\ep^{\frac{d}{2} - 1} f_-\Ep)_{\ep > 0}$ and $(\ep^{\frac{d}{2} - 1} f_-\Ep)_{\ep > 0}$ weakly star converge to some probability measures $f_-$ and $f_+$ on $\Ob$ :
\[
\lim_{\ep \rightarrow 0^+} \ep^{d/2-1}\sum_{x \in N\Ep} ( \vp(x) f_-\Ep(x) + \psi(x) f_+\Ep(x)) = \IOb \vp df_- + \IOb \psi df_+, \: \forall (\vp, \psi) \in C(\Ob)^2.
\]
\end{hyp}

\begin{hyp} \label{hy5}
There exists $N \in \NN, \{v_k\}_{k=1,\dots,N} \in C^1(\RR^d,\Ss)^N$ and $\{c_k\}_{k=1,\dots,N} \in C^1(\Ob, \RR_+^*)^N$ such that $E^{\ep}$ weakly converges in the sense that
\[
\lim_{\ep \rightarrow 0^+} \sum_{(x,e) \in E^{\ep}} |e|^d \varphi \left (x, \frac{e}{|e|} \right ) = \int_{\OS} \varphi(x,v) \: \theta(dx,dv), \forall \varphi \in C(\ObS),
\]
where $\theta \in \mathcal{M}_+(\Omega \times \Ss)$ and $\theta$ is of the form
\[
\theta(dx,dv)= \sum_{k=1}^N c_k(x) \delta_{v_k(x)} dx.
\]
Moreover, there exists a constant $C >0$ such that for every $(x,z, \xi) \in \RR^d \times \Ss \times \RR_+^N$, there exists $\bar{Z} \in \RR_+^N$ such that $|\bar{Z}| \leq C$ and
\beq \label{min}
\bar{Z} \cdot \xi = \min \left \{Z \cdot \xi; Z = (z_1, \dots, z_N) \in \RR_+^N \text{ and } \sum_{k=1}^N z_k v_k(x) = z \right \}.
\eeq
\end{hyp}

The $c_k$'s are the volume coefficients and the $v_k$'s are the directions in the network.
The next assumption focuses on the congestion functions $g\Ep$.
 \begin{hyp} \label{hy2}
 $g\Ep$ is of the form
 \begin{equation} \label{H2}
 g^{\ep}(x,e,m)= |e|^{d/2} g \left (x, \frac{e}{|e|}, \frac{m}{|e|^{d/2}} \right), \: \forall \ep >0, (x,e) \in \Eep, m \geq 0
 \end{equation}
 where $g : \OS \times \mathbb{R}_+ \mapsto \mathbb{R}$ is a given continuous, nonnegative function that is increasing in its last variable.
 \end{hyp}

We then have
  \[
 G^{\ep}(x,e,m)= |e|^d G \left (x, \frac{e}{|e|}, \frac{m}{|e|^{d/2}} \right)  \text{ where } G(x, v, m) := \int_0^m g(x, v, \alpha) d \alpha.
 \]

 We also add assumptions on $G$:
 \begin{hyp} \label{hy3}
 There exists a closed neighborhood $U$ of $\Ob$ such that for $k=1, \dots, N$, $v_k$ may be extended on U in a function $C^1$ (still denoted $v_k$). Moreover, each function $(x,m) \in U \times \RR_+ \mapsto G(x,v_k(x),m)$ is Carath\'{e}odory, convex nondecreasing in its second argument with $G(x, v_k(x), 0) = 0$ a.e. $x \in U$ and there exists $1<q<d/(d-1)$ and two constants $0 < \lambda \leq \Lambda$ such that for every $(x,m) \in U \times \RR_+$ one has
 \begin{equation} \label{H3}
 \lambda (m^q -1) \leq G(x,v, m) \leq \Lambda (m^q +1).
 \end{equation}
 \end{hyp}

 The $q$-growth is natural since we want to work in $L^q$ in the continuous limit. The condition on $q$ has a technical reason. It means that the conjugate exponent $p$ of $q$ is $>d$, which allows us to use  Morrey's inequality in the proof of the convergence (\cite{Romeotheboss}). The extension on $U$ will serve to use regularization by convolution and Moser's flow argument. Examples of models that satisfy these assumptions are regular decompositions. In two-dimensional networks, there exists three different regular decompositions: cartesian, triangular and hexagonal. In these models, the length of an arc in $\Eep$ is $\ep$. The $c_k$'s and $v_k$'s are constant. In the cartesian case, $N=4$, $(v_1, v_2, v_3, v_4) := ((1,0),(0,1),(-1,0),(0,-1))$ and $c_k = 1$ for $k=1, \dots, 4$. For more details, see \cite{Romeotheboss}.

  Now, before presenting the continuous limit problem, let us set some notations.

  Let us write the set of generalized curves
\[
\mathcal{L} = \{ (\gamma, \rho) : \gamma \in W^{1,\infty}([0,1], \Ob), \rho \in \mathcal{P}_{\gamma} \cap L^1([0,1])^N\},
\]
  where
  \[
  \mathcal{P}_{\gamma} = \left \{ \rho : t \in [0,1] \rightarrow \rho(t) \in \RR_+^N \text{ and } \dot{\gamma}(t) = \sum_{k=1}^N v_k(\gamma(t)) \: \rho_k (t) \text{ a.e.} \right \}.
\]
We can notice that $\Pc_{\gamma}$ is never empty thanks to \ref{hy5}.
Let us denote $Q \in \mathcal{Q}(f_-, f_+)$ the set of Borel probability measures $Q$ on $\Lc$ such that the mass conservation constraints are satisfied
\[
\mathcal{Q}(f_-, f_+) := \{ Q \in \mathcal{M}^1_+(\Lc) : {e_0}_{\#} Q = f_-, {e_1}_{\#}Q = f_+\}
\]
where $e_t(\gamma, \rho) = \gamma(t), t \in [0,1], (\gamma, \rho) \in \Lc$. For $k=1,\dots,N$ let us then define the nonnegative measures on $\ObS$, $m_k^Q$ by
\beq \label{mQ}
\int_{\ObS} \varphi(x,v) dm_k^Q(x,v) = \int_{\mathcal{L}} \left ( \int_0^1 \varphi(\gamma(t),v_k(\gamma(t))) \rho_k(t) dt \right ) dQ(\gamma, \rho),
\eeq
for every $\varphi \in C(\ObS, \mathbb{R}).$ Then write simply $m^Q = \sum_{k=1}^N m_k^Q$, nonnegative measure on $\ObS$.
Finally assume that
\[
\mathcal{Q}^q(f_-, f_+) := \{ Q \in \mathcal{Q}(f_-, f_+): m^Q \in L^q(\theta) \} \neq \emptyset.
\]

It is true when for instance, $f_+$ and $f_-$ are in $L^q(\Omega)$ and $\Omega$ is convex. Indeed, first for $Q \in \Mc_+^1(W^{1,\infty}([0,1], \Ob))$, let us define $i_Q \in \Mc_+(\Ob)$ as follows
\[
\IO \vp \: di_Q = \int_{W^{1,\infty}([0,1], \Ob)} \left ( \int_0^1 \vp(\gamma(t)) |\dot{\gamma}(t)| dt \right ) dQ(\gamma) \text{ for } \vp \in C(\Ob, \RR).
\]
It follows from the regularity results of \cite{de2004integral,santambrogio2009absolute} that there exists $Q \in \Mc_+^1(W^{1,\infty}([0,1], \Ob))$ such that ${e_0}_{\#} Q = f_-$, ${e_1}_{\#}Q = f_+$ and $i_Q \in L^q$. For each curve $\gamma$, let $\rho^{\gamma} \in \Pc_{\gamma}$ such that $\sum_k \rho_k^{\gamma}(t) \leq C |\dot{\gamma}(t)|$ (we have the existence due to \ref{hy5}). Then we set $\tilde{Q} = {(id, \rho^{\cdot})}_{\#}Q$. We have $\tilde{Q} \in \mathcal{Q}^q(f_-, f_+)$ so that we have proved the existence of such kind of measures.

Then Wardrop equilibria at scale $\ep$ converge as $\ep \rightarrow 0^+$ to solutions of the following problem
\begin{equation} \label{pdb}
\inf_{Q \in \mathcal{Q}^q(f_-, f_+)} \IOS G (x, v, m^Q(x, v)) \theta(dx, dv)
\end{equation}
(see \cite{Romeotheboss}). Nevertheless this problem \eqref{pdb} is posed over probability measures on generalized curves and it is not obvious at all that it is simpler to solve than the discrete problem \eqref{P1}. So in the present paper, we want to show that problem \eqref{pdb} is equivalent to another problem that will roughly amount to solve an elliptic PDE. This problem is
\begin{equation} \label{pb2}
\inf_{\sigma \in L^q(\Omega, \mathbb{R}^d)}
\inf_{\varrho \in \mathcal{P}^{\sigma}} \left \{ \int_{\OS} G(x,v, \varrho(x,v)) \: \theta (dx, dv); \: -\text{div } \sigma = f \right \},
\end{equation}
where
\begin{equation*}
\mathcal{P}^{\sigma} = \left \{ \varrho : \Omega \times \Ss \rightarrow \RR_+; \: \forall x \in \Omega, \: \sigma(x) = \sum_{k=1}^N v_k(x) \varrho(x, v_k(x)) \right \},
\end{equation*}
$f=f_+ - f_-$ and the equation $-\text{div}(\sigma) = f$ is defined by duality:
\[
\IO \nabla u \cdot \sigma = \IO u \: df, \text{ for all } u \in C^1(\Ob),
\]
so the homogeneous Neumann boundary condition $\sigma \cdot \nu_{\Omega} = 0$ is satisfied on $\partial \Omega$ in the weak sense. For the sake of clarity, let us define
\[
\mathcal{G}(x, \sigma) := \inf_{\vr \in \mathcal{P}_x^{\sigma}} \sum_{k=1}^N c_k(x) G(x,v_k(x), \vr_k)  := \inf_{\vr \in \mathcal{P}_x^{\sigma}} \overline{G}(x, \vr)
\]
where
\[
\mathcal{P}_x^{\sigma} = \left \{ \vr \in \RR_+^N; \: \sigma = \sum_{k=1}^N v_k(x) \vr_k \right \} \text{ and }  \overline{G}(x, \vr) := \sum_{k=1}^N c_k(x) G(x,v_k(x), \vr_k),
 \]
 for $x \in \Omega, \sigma \in \RR^d$. We recall that the $c_k$'s are the volume coefficients in $\theta$. $\Gc$ is convex in the second variable (since $G$ is convex in its last variable).The minimization problem \eqref{pb2} can then be rewritten as
\begin{equation}
\label{PP}
\inf_{\sigma \in L^q(\Omega, \mathbb{R}^d)}
 \left \{ \int_{\Omega} \mathcal{G}(x,\sigma(x)) \: dx; \: -\text{div } \sigma = f \right \}.
\end{equation}

This problem \eqref{PP} looks like the ones introduced by Beckmann \cite{beckmann1952continuous} for the design of an efficient commodity transport program. The dual problem of \eqref{PP} takes the form
\begin{equation}
\label{PD}
\sup_{u \in W^{1,p}(\Omega)} \left \{ \int_{\Omega} u \: df - \int_{\Omega} \mathcal{G}^*(x, \nabla u(x)) \: dx \right \},
\end{equation}
where $p$ is the conjugate exponent of $q$ and $\Gc^*$ is the Legendre transform of $\Gc(x, \cdot)$. In order to solve \eqref{PP}, we can first solve the Euler-Lagrange equation of its dual formulation and then use the primal-dual optimality conditions. Nevertheless, in our typical congestion models, the functions $G(x,v,\cdot)$ have a positive derivative at zero (that is $g(x,v,0)$). Indeed, going at infinite speed - or teleportation - is not possible even when there is no congestion. So we have a singularity in the integrand in \eqref{PP}. Then $G^*$ and the Euler-Lagrange equation of \eqref{PD} are extremely degenerate. Moreover, the prototypical equation of \cite{brasco2010congested} is the following
\[
-\text{div} \left ( (|\nabla u| - 1)^{p-1}_+ \frac{\nabla u}{|\nabla u|} \right ) = f.
\]
Here, for well chosen $g$, we obtain anisotropic equation of the form
\[
- \sum_{k=1}^N \sum_{l=1}^d \partial_l \left [ b_k(x) v_{kl}(x) (\nabla u \cdot v_k(x) - \delta_k c_k(x))_+^{p-1} \right ]  = f.
\]
where $v_k(x) = (v_{k1}(x), \dots, v_{kd}(x))$ for $k=1, \dots, N$ and $x \in \Ob$. In the cartesian case, we can separate the variables in the sum but in the hexagonal one ($d=2$), it is  impossible. The previous equation degenerates in an unbounded set of values of the gradient and its study is delicate, even if all the $\delta_k$'s are zero. It is more complicated than the one in \cite{brasco2013congested}. Indeed, the studied model in \cite{brasco2013congested} is the cartesian one and the prototypical equation is
\[
- \sum_{k=1}^2 \partial_k \left ( (|\partial_k u| - \delta_k)_+^{p-1} \frac{\partial_k u}{|\partial_k u|} \right ) = f.
\]
The plan of the paper is as follows. In Section 2, we formulate some relationship between \eqref{pdb} and \eqref{PP}. Section 3 is devoted to optimality conditions for \eqref{PP} in terms of solutions of \eqref{PD}. We also present the kind of PDEs that represent realistic anisotropic models of congestion. In Section 4, we give some regularity results in the particular case where the $c_k$'s and the $v_k$'s are constant. Finally, in Section 5, we describe numerical schemes that allow us to approximate the solutions of the PDEs.

\section{Equivalence with Beckmann problem}

Let us study the relationship between problems \eqref{pdb} and \eqref{pb2}. We still assume that all specified hypothesis in \ref{intro} are satisfied. Let us notice that thanks to \ref{hy5}, for every $\sigma \in L^q(\Omega, \mathbb{R}^d)$, there exists $\hat{\varrho} \in \mathcal{P}^{\sigma}$ such that $\hat{\vr} \in L^q(\theta)$ and $\hat{\vr}$ minimizes the following problem :
\[
\inf_{\varrho \in \mathcal{P}^{\sigma}} \left \{ \int_{\OS} G(x,v, \varrho(x,v)) \: \theta (dx, dv) \right \}.
\]
For $\varrho \in \Pc^{\sigma}$, define $\bar{\varrho} : \Omega \rightarrow \RR_+^N$ where $\bar{\varrho}_k(x) = \varrho(x, v_k(x))$, for every $x \in \Omega, k=1, \dots, N.$ Now, we only consider $\bar{\varrho}$ that we simply write $\varrho$ (by abuse of notations).

\begin{theo}
Under all previous assumptions, we have
\begin{equation*}
\inf \eqref{pdb} = \inf \eqref{PP}.
\end{equation*}
\end{theo}

\begin{proof} We adapt the proof in \cite{brasco2013congested}. We will show the two inequalities. \\
\textbf{Step 1: } $ \inf \eqref{pdb} \geq \inf \eqref{PP}$. \\
Let $Q \in \mathcal{Q}^q(f^+,f^-)$. We build $\sigma^Q \in L^q(\Omega, \mathbb{R}^d)$ that will allow us to obtain the desired inequality, we define it as follows :
\beq \label{sigmaQ}
\IO \varphi \: d\sigma^Q = \int_{\mathcal{L}} \int_0^1 \vp (\gamma(t)) \cdot \dot{\gamma}(t) dt \: dQ(\gamma,\rho), \forall \varphi \in C(\Ob, \RR^d).
\eeq
In particular, we have that $-\text{div }\sigma^Q = f$ since $Q \in \Qc(f_-, f_+)$.
We now justify that
\[
\sigma^Q(x) = \int_{\Ss} v \, m^Q(x,v) \: dv = \sum_{k=1}^N v_k(x) m^Q(x,v_k(x)) \: \text{ a.e. } x \in \Ob.
\]

Recall that for every $\xi \in C(\ObS, \mathbb{R}),$
\[
\int_{\ObS} \xi dm^Q = \int_{\mathcal{L}}  \int_0^1 \left ( \sum_{k=1}^N \xi(\gamma(t),v_k(\gamma(t))) \rho_k(t) \right ) dt \: dQ(\gamma, \rho).
\]

By taking $\xi$ of the form $\xi(x,v)= \varphi(x) \cdot v$ with $\varphi \in C(\overline{\Omega}, \mathbb{R}^d)$, we get
\begin{align*}
\int_{\ObS} \varphi(x) \cdot v \: dm^Q(x,v) & = \int_{\mathcal{L}}  \int_0^1  \left ( \sum_{k=1}^N \rho_k(t) \vp(\gamma(t)) \cdot v_k(\gamma(t)) \right ) dt \: dQ(\gamma, \rho) \\
& = \int_{\Omega} \varphi \: d \sigma^Q.
\end{align*}
Moreover, since $m^Q \geq 0$, we obtain that $m^Q \in \mathcal{P}^{\sigma^Q}$ (and so that $\sigma^Q \in L^q$) and the desired inequality follows. \\

\textbf{Step 2: } $ \inf \eqref{pdb} \leq \inf \eqref{PP}$. \\
Now prove the other inequality. We will use Moser's flow method (see \cite{brasco2010congested, dacorogna1990partial, moser1965volume}) and a classical regularization argument. Fix $\delta >0$. Let $\sigma \in L^q(\Omega, \mathbb{R}^d)$ and $\vr \in \mathcal{P}^{\sigma} \cap L^q(\Omega, \RR^N)$ such that
\[
\int_{\OS} G(x,v, \varrho(x,v)) \: \theta (dx, dv) \leq \text{inf } \eqref{PP} + \delta
\]
with $-\text{div } \sigma = f$. We extend them outside $\Omega$ by $0$. Let then $\eta \in C_c^{\infty} (\mathbb{R}^d)$ be a positive function, supported in the unit ball $B_1$ and such that $\int_{\mathbb{R}^d} \eta = 1$.  For $\ep  \ll 1$ so that $\Oep := \Omega + \ep B_1 \Subset U$, we define $\eta\Ep(x) := \ep^{-d} \eta (\ep^{-1} x)$, $\sigma^{\ep} := \eta\Ep \star \sigma$ and $\vr_k\Ep(x) := \eta\Ep \star \vr_k(x)$ for $k=1, \dots, N$. By construction, we thus have that $\sigma^{\ep} \in C^{\infty} (\overline{\Omega}_{\ep})$ and
\[
-\text{ div } (\sigma^{\ep}) = f_+^{\ep} - f_-^{\ep} \: \text{ in } \Oep \: \text{ and } \: \sigma^{\ep} = 0 \: \text{ on } \: \partial \Oep,
\]
where $f_{\pm}^{\ep} = \eta\Ep \star (f_{\pm} 1_{\overline{\Omega}} ) + \ep$. But the problem is that we do not have $\vr\Ep \in \mathcal{P}^{\sigma\Ep}$. We shall build a sequence $(P\Ep)$ in $\mathcal{P}^{\sigma\Ep}$ that converges to $\rho$ in $L^q(U, \RR^N)$. Notice that
\begin{align*}
\sigma\Ep(x) & = \sum_{k=1}^N \int \eta\Ep(y) \vr_k(x-y) v_k(x-y) \: dy \\
& = \sum_{k=1}^N \vr\Ep_k(x) v_k(x) + \sum_{k=1}^N \int \eta\Ep(y) \vr_k(x-y) (v_k(x-y)-v_k(x)) \: dy
\end{align*}
There exists $p\Ep_k \in L^q(\Oep)$ such that for every $k=1, \dots, N,$ $p\Ep_k \geq 0, p\Ep_k \rightarrow 0$ and for $x \in \Oep$, we have
\[
I\Ep(x) = \sum_{k=1}^N \int \eta\Ep(y) \vr_k(x-y) (v_k(x-y)-v_k(x)) \: dy = \sum_{k=1}^N p\Ep_k(x) v_k(x).
\]
 Such a family exists since $I\Ep \in L^q$ and $I\Ep \rightarrow 0$ (by using the fact that the $v_k$'s are in $C^1(U)$) and we can estimate $p\Ep_k$ with $I\Ep$ due to \ref{hy5}. Then if we set $P\Ep = \vr\Ep + p\Ep$, we have $P\Ep \in \mathcal{P}^{\sigma\Ep}$ and $P\Ep \rightarrow \vr$ in $L^q$.

Define $g^{\ep}(t,x) := (1-t) f_-^{\ep}(x) + t f_+^{\ep} (x)  \, \forall t \in [0,1], \: x \in \Ob_{\ep}$, let then $X^{\ep}$ be the flow of the vector field $v^{\ep} := \sigma^{\ep} / g^{\ep}$, that is,
\[
\left \{ \begin{aligned}
& \dot{X}_t^{\ep} (x) = v^{\ep} (t, X_t^{\ep}(x)) \\
& X_0^{\ep} (x) =x, \: \: (t,x) \in [0,1] \times \Ob_{\ep}.
\end{aligned} \right.
\]
We have $\partial_t g\Ep + \text{div } (g\Ep v\Ep) = 0$. Since $v\Ep$ is smooth and the initial data is $g\Ep (0, \cdot) = f_-\Ep$, we have ${X_t\Ep}_{\#}f_-\Ep = g\Ep(t, \cdot)$. Let us define the set of generalized curves
\[
 \Lc_{\ep} = \{ (\gamma, \rho) : \gamma \in W^{1,\infty}([0,1], \Oep), \rho \in \mathcal{P}_{\gamma} \cap L^1([0,1])^N\}.
 \]
Let us consider the following measure $Q^{\ep}$ on $\Lc_{\ep}$
\[
Q^{\ep} = \int_{\Ob_{\ep}} \delta_{(X_{\cdot}^{\ep}(x), P\Ep( X_{\cdot}^{\ep}(x)) / g^{\ep}(\cdot, X_{\cdot}^{\ep}(x) )) } df_-^{\ep}(x).
\]
We then have  ${e_t}_{\#}Q^{\ep} = {X_t\Ep}_{\#}f_-\Ep = g\Ep(t, \cdot)$ for $t \in [0,1]$. We define $\sigma^{Q\Ep}$ and $m_k^{Q\Ep}$ as in \eqref{sigmaQ} and \eqref{mQ} respectively, by using test-functions defined on $\Oep$.
We then have $\sigma^{Q^{\ep}}= \sigma^{\ep}$. Indeed, for $\varphi \in C(\Ob_{\ep}, \mathbb{R}^d)$, we have
\begin{align*}
\int_{\Ob_{\ep}} \varphi \: d \sigma^{Q^{\ep}}
&  = \int_{\Ob_{\ep}} \int_0^1 \varphi (X_t^{\ep}(x)) \cdot v^{\ep}(t, X_t^{\ep}(x)) f_-^{\ep}(x) \: dt \: dx \\
&  = \int_0^1 \int_{\Ob_{\ep}}  \varphi (x) \cdot v^{\ep}(t, x) g^{\ep}(t,x) \: dx \: dt \\
& = \int_{\Ob_{\ep}} \varphi \: d \sigma^{\ep}
\end{align*}
which gives the equality. We used the definition of $Q\Ep$, the fact that ${X_t\Ep}_{\#}f_-\Ep = g\Ep(t, \cdot)$ and that $v\Ep g\Ep = \sigma\Ep$ and Fubini's theorem.
In the same way, we have $m^{Q^{\ep}} \in \mathcal{P}^{\sigma^{\ep}}$. To prove it, we take the same arguments as in the end of Step 1 and in the previous calculation. For $\vp \in C(\Oep, \RR^d)$, we have
\begin{align*}
\int_{\Oep \times \Ss} \vp(x) & \cdot v \: m^{Q\Ep}(dx, dv) \\
& = \int_0^1 \left ( \int_{\Oep} \sum_{k=1}^N \vp(X_t^{\ep}(x)) \cdot v_k(X_t^{\ep}(x)) \frac{P_k\Ep(X_t^{\ep}(x))}{g\Ep(t,X_t^{\ep}(x))} f_-\Ep(x) dx \right ) dt \\
& = \int_0^1 \left ( \int_{\Oep} \vp(X_t^{\ep}(x)) \cdot \frac{\sigma\Ep(X_t^{\ep}(x))}{g\Ep(t,X_t^{\ep}(x))} f_-\Ep(x) dx \right ) dt \\
& = \int_0^1 \left ( \int_{\Oep} \vp(x) \cdot \sigma\Ep(x) dx \right ) dt \\
& = \int_{\Oep} \vp \: d\sigma\Ep.
\end{align*}
Moreover, more precisely, we have $m_k^{Q\Ep} (dx,dv) = \delta_{v_k(x)} P\Ep_k(x) dx$. Then we conclude as in \cite{brasco2013congested}. First for any Lipschitz curve $\vp$, let us denote by $\tilde{\vp}$ its constant speed reparameterization, that is, for $t \in [0,1], \tilde{\vp}(t) = \vp( s^{-1}(t)),$ where
\[
s(t) = \frac{1}{l(\vp)}\int_0^t |\dot{\vp}(u)| du \text{ with } l(\vp) = \int_0^1 |\dot{\vp}(u)| du.
\]
For $(\vp, \rho) \in \Lc$, let $\tilde{\rho}$ be the reparameterization of $\rho$ i.e.
\[
\tilde{\rho}_k(t) := \frac{l(\sigma)}{| \dot{\sigma}(s^{-1}(t))|} \rho_k(s^{-1}(t)), \forall t \in [0,1], k=1, \ldots, N.
\]
Let us denote by $\tilde{Q}$ the push forward of $Q$ through the map $(\vp, \rho) \mapsto (\tilde{\vp}, \tilde{\rho})$. We have $m_k^{\tilde{Q}}=m_k^Q$ and $\sigma^{\tilde{Q}}=\sigma^Q$. Then arguing as in \cite{Romeotheboss}, the $L^q$ bound on $m^{Q\Ep}$ yields the tightness of the family of Borel measures $\tilde{Q}\Ep$ on $C([0,1], \RR^d) \times L^1([0,1])^N$. So $Q^{\ep}$ $\star$-weakly converges to some measure $Q$ (up to a subsequence). Let us remark that $\tilde{Q}\Ep$ has its total mass equal to that of $f_+\Ep$, that is, $1+\ep |\Oep|$. Thus one can show that $Q(\Lc) = 1$) (due to the fact that $Q(\Lc)=\lim_{\ep \rightarrow 0^+} Q(\Lc_{\ep}) = 1$). Moreover, we have $Q \in \Qc(f_-, f_+)$ thanks to the $\star$-weak convergence of $\tilde{Q}\Ep$ to $Q$. Recalling the fact that $P\Ep_k = m^{Q^{\ep}}(\cdot,v_k(\cdot)) $ strongly converges in $L^q$ to $\vr_k$ ($\vr \in \Pc^{\sigma}$) and due to the same semicontinuity argument as in \cite{carlier2008optimal, Romeotheboss}, we have $m^Q(\cdot, v_k(\cdot)) \leq \vr_k$ in the sense of measures. Then $m^Q(\cdot, v_k(\cdot)) \in L^q$ so that $Q \in \Qc^q(f_-, f_+)$. It follows from the monotonicity of $G(x,v,\cdot)$ that :
\begin{align*}
\int_{\OS} G(x,v, m^Q(x,v)) \: \theta (dx, dv) & \leq \int_{\OS} G(x,v, \vr(x,v)) \: \theta (dx, dv) \\
& \leq \inf \eqref{PP} + \delta.
\end{align*}
Letting $\delta \rightarrow 0^+$, we have the desired result.
\end{proof}

In fact, we showed in the previous proof a stronger result. We proved the following equivalence
\[
Q \text{ solves } \eqref{pdb} \Longleftrightarrow \sigma^Q \text{ solves } \eqref{pb2}
\]
and moreover,
\[
(m^Q(\cdot, v_k(\cdot)))_{k=1, \dots, N} \in \mathcal{P}^{\sigma^Q}
\]
is optimal for \eqref{pdb}. We also built a minimizing sequence for \eqref{pdb} from a regularization of a solution $\sigma$ of \eqref{pb2} by using Moser's flow argument.

\section{ Characterization of minimizers via anisotropic elliptic PDEs}

Here, we study the primal problem \eqref{PP} and its dual problem \eqref{PD}. Recalling that $f=f_+ - f_-$ has zero mean, we can reduce the problem \eqref{PD} only to zero-mean $W^{1,p}(\Omega)$ functions. Since for $(x, v) \in \ObS$ and $k=1,\dots,N$, $G(x,v, \cdot)$ has a positive derivative at zero, $G$ is strictly convex in its last variable then so is $\Gc(x, \cdot)$ for $x \in \Ob$. Thus $\Gc^*$ is $C^1$. However $\Gc$ is not differentiable so that $\Gc^*(x, \cdot)$ is degenerate.
By standard convex duality (Fenchel-Rockafellar's theorem, see \cite{ekeland1976convex} for instance), we have that $\min \eqref{PP} = \max \eqref{PD}$ and we can characterize the optimal solution $\sigma$ of \eqref{PP} (unique, by strict convexity) as follows
\[
\sigma(x)= \nabla \mathcal{G}^*(x, \nabla u(x)),
\]
where $u$ is a solution of \eqref{PD}.
In other terms, $u$ is a weak solution of the Euler-Lagrange equation
\[
\left \{
\begin{aligned}
 -\text{ div } (\nabla  \mathcal{G}^*(x, \nabla u(x))) & = f \: & \text{ in } \Omega, \\
   \nabla  \mathcal{G}^*(x, \nabla u(x)) \cdot \nu_{\Omega} & = 0 \: & \text{ on } \partial \Omega,
\end{aligned} \right.
\]
in the sense that
\[
\int_{\Omega} \nabla  \mathcal{G}^*(x, \nabla u(x)) \cdot \nabla \varphi (x) \: dx = \int_{\Omega} \varphi(x) \: df(x) , \: \forall \varphi \in W^{1,p}(\Omega).
\]
Let us remark that if $u$ is not unique, $\sigma$ is.

A typical example is $g(x, v_k(x), m) = g_k(x, m) = a_k(x) m^{q-1} + \delta_k$ with $\delta_k > 0$ and the weights $a_k$ are regular and positive.  We can explicitly compute $\mathcal{G}^*(x, z)$. Let us notice that for every $x \in \Omega, z \in \RR^d$, we have :
\begin{align*}
\mathcal{G}^*(x, z) & =  \sup_{\sigma \in \RR^d} (z \cdot \sigma - \mathcal{G}(x, \sigma) ) =  \sup_{\sigma \in \RR^d} (z \cdot \sigma - \inf_{\vr \in \mathcal{P}_x^{\sigma}} \overline{G}(x, \vr) ) \\
 & = \sup_{\sigma, \vr} (z \cdot \sigma - \overline{G} (x, \vr)) = \sup_{\vr \in \RR_+^N } \left \{ \sum_{k=1}^N ( z \cdot v_k(x)) \vr_k - \overline{G} (x, \vr) \right \}.
 \end{align*}
A direct calculus then gives
 \[
 \mathcal{G}^*(x, z) = \sum_{k=1}^N \frac{b_k(x)}{p} (z \cdot v_k(x) - \delta_k c_k(x) )_+^p,
  \]
where $b_k = (a_k c_k)^{-\frac{1}{q-1}}$. The PDE then becomes
\begin{equation} \label{3.4}
- \sum_{k=1}^N \sum_{l=1}^d \partial_l \left [ b_k(x) v_{kl}(x) (\nabla u \cdot v_k(x) - \delta_k c_k(x) )_+^{p-1} \right ]  = f,
\end{equation}
where $v_k(x) = (v_{k1}(x), \dots, v_{kd}(x))$.

For $k=1, \dots, N$,  $\mathcal{G}_k^*(x, z) = \frac{b_k(x)}{p} (z \cdot v_k(x) - \delta_k )_+^p$ vanishes if $z \cdot v_k(x) \in ]-\infty, \delta_k c_k(x)]$ so that any $u$ whose the gradient satisfies $\nabla u(x) \cdot v_k(x) \in ]-\infty, \delta_k c_k(x) ], \forall x \in \Omega, k=1,\dots, N$ is a solution of the previous PDE with $f=0$. In consequence, we cannot hope to obtain estimates on the second derivatives of $u$ or even oscillation estimates on $\nabla u$ from \eqref{3.4}. Nevertheless we will see that we have some regularity results on the vector field $\sigma = (\sigma_1, \dots, \sigma_d)$ that solves \eqref{PP} in the case where the directions and the volume coefficients are constant, that is,
\[
\sigma (x) = \sum_{k=1}^N \left [ b_k(x) (\nabla u (x) \cdot v_k - \delta_k c_k )_+^{p-1} \right ] v_k,
\]
for every $x \in \Omega$.

\section{Regularity when the $v_k$'s and $c_k$'s are constant}

Our aim here is to get some regularity results in the case where the $v_k$'s and the $c_k$'s are constant. We will strongly base on \cite{brasco2013congested} to prove this regularity result. Let us consider the model equation
\begin{equation} \label{4.10}
- \sum_{k=1}^N \text{div} \left ( (\nabla u(x) \cdot v_k - \delta_k c_k )_+^{p-1} v_k \right )  = f,
\end{equation}
where $v_k \in \Ss, c_k > 0$ and $b_k \equiv 1$ for $k=1, \dots, N$. Define for $z \in \RR^d$
\beq
F (z) = \sum_{k=1}^N F_k(z), \text{ with } F_k(z) =  (z \cdot v_k - \delta_k c_k)_+^{p-1}  v_k
\eeq
and
\beq
H (z) = \sum_{k=1}^N H_k(z), \text{ with } H_k(z) =  (z \cdot v_k - \delta_k c_k)_+^{\frac{p}{2}} v_k.
\eeq
Here we assume only $p \geq 2$. We have the following lemma that establishes some connections between $F$ and $H$.

\begin{lemme} \label{lem4.2}
Let $F$ and $G$ be defined as above with $p \geq 2$, then for every $(z,w) \in \RR^d \times \RR^d$, the following inequalities are true for $k=1, \dots, N$
 \beq \label{eq2}
 |F_k(z)| \leq |z|^{p-1},
 \eeq
 \beq \label{eq3}
|F_k(z) -F_k(w)| \leq (p-1) \left ( |H_k(z) |^{\frac{p-2}{p}} + |H_k(z) |^{\frac{p-2}{p}} \right ) |H_k(z) - H_k(w)|,
\eeq
and
\beq \label{eq4}
 (F_k(z) - F_k(w)) \cdot (z - w) \geq \frac{4}{p^2} |H_k(z) - H_k(w) |^2.
 \eeq
\end{lemme}

\bpr
The first one is trivial.
For the second one, from \cite{Lindqvist} one has the general result: for all $(a,b) \in \RR^d \times \RR^d$, the following inequality holds
\beq \label{4.13}
\left | |a|^{p-2} a - |b|^{p-2} b \right | \leq (p-1) \left ( |a|^{\frac{p-2}{2}} + |b|^{\frac{p-2}{2}} \right ) \left | |a|^{\frac{p-2}{2}} a - |b|^{\frac{p-2}{2}} b \right |.
\eeq
Choosing $a = (z \cdot v_k - \delta_k c_k)_+ v_k$ and $b = (w \cdot v_k - \delta_k c_k )_+ v_k$ in \eqref{4.13}, we then obtain \eqref{eq3}.

Let us now prove the third inequality. It is trivial if both $z \cdot v_k$ and $w \cdot v_k$ are less than $\delta_k c_k$. If $z \cdot v_k > \delta_k c_k$ and $w \cdot v_k \leq \delta_k c_k$, we have
\[
 (F_k(z) - F_k(w)) \cdot (z - w) = (z \cdot v_k - \delta_k c_k)_+^{p-1} ( z \cdot v_k - w \cdot v_k) \geq (z \cdot v_k - \delta_k c_k)_+^p = |H_k(z)|^2.
 \]
 For the case $z \cdot v_k > \delta_k c_k$ and $w \cdot v_k > \delta_k c_k$, we use the following inequality (again \cite{Lindqvist})
 \[
 (|a|^{p-2} a - |b|^{p-2} b ) \cdot (a-b) \geq \frac{4}{p^2} \left ( |a|^{\frac{p-2}{2}} a - |b|^{\frac{p-2}{2}} b \right )^2.
  \]
  Again taking $a = (z \cdot v_k - \delta_k c_k)_+ v_k$ and $b = (w \cdot v_k - \delta_k c_k)_+ v_k$, we have that
\bal
  \frac{4}{p^2} |H_k(z) & - H_k(w) |^2 \\
  & \leq (|F_k(z)| - |F_k(w)|)v_k \cdot ((z \cdot v_k - \delta_k c_k)_+ - (w \cdot v_k - \delta_k c_k )_+) v_k \\
  & = (|F_k(z)| - |F_k(w)|) (z-w) \cdot v_k,
 \end{align*}
 which gives \eqref{eq4}.
   \epr

Let us fix $f \in W_{\loc}^{1,q}(\Omega)$ where $q$ is the conjugate exponent of $p$ and let us consider the equation
\begin{equation} \label{4.7}
-\text{div} F(\nabla u) = f.
\end{equation}

Thanks to Nirenberg's method of incremental ratios, we then have the following result that is strongly inspired of Theorem $4.1$ in \cite{brasco2013congested}:
\begin{theo} \label{thmreg}
Let $u \in W_{\loc}^{1,p}(\Omega)$ be a local weak solution of \eqref{4.7}. Then $\Hc := H(\nabla u) \in W_{\loc}^{1,2}(\Omega)$. More precisely, for every $k=1, \dots, N, \Hc_k :=  H_k(\nabla u) \in W_{\loc}^{1,2}(\Omega)$.
\end{theo}

\bpr
For the sake of clarity, write $\Fc := F(\nabla u)$ and similarly, $\Fc_k, \Hc_k$ (note that $\Fc_k \in L_{\loc}^q(\Omega)$ and $\Hc_k \in L_{\loc}^2(\Omega)$ due to \eqref{eq2}-\eqref{eq3}. Let us define the translate of the function $\vp$ by the vector $h$ by $\tau_h \vp := \vp (\cdot + h)$. Let $\vp \in W^{1,q}(\Omega)$ be compactly supported in $\Omega$ and $h \in \RR^d \backslash \{0\}$ be such that $|h| < \text{dist(supp}(\vp), \RR^d \backslash \{0\})$, we then have
\beq \label{eq9}
\IO \frac{\tau_h \Fc - \Fc}{|h|} \cdot \nabla \vp dx = \IO \frac{\tau_h f - f}{|h|} \cdot \vp dx.
\eeq

Let $\omega \Subset \omega_0 \Subset \Omega$ and $\xi \in C_c^{\infty}(\Omega)$ such that supp($\xi$) $\subset \omega_0, 0 \leq \xi \leq 1$ and $\xi = 1$ on $\overline{\omega}$ and $h \in \RR^d \backslash \{0\}$ such that $|h| \leq r_0 < \frac{1}{2} \text{dist}(\omega_0, \RR^d \backslash \Omega)$. In what follows, we denote by $C$ a nonnegative constant that does not depend on $h$ but may change from one line to another. We then introduce the test function
\begin{equation*}
\vp = \xi^2 |h|^{-1} (\tau_h u -u),
\end{equation*}
in \eqref{eq9}. Let us fix $\omega' := \omega_0 + B(0, r_0)$. It follows from $u \in W_{\loc}^{1,p}(\Omega), f \in W_{\loc}^{1,q}(\Omega)$ and the H\"older inequality that
\begin{equation*}
|h|^{-2} \IO (\tau_h \Fc - \Fc) \cdot \left ( \xi^2 (\tau_h \nabla u - \nabla u) + 2 \xi \nabla \xi (\tau_h u - u) \right ) \leq \| \nabla f \|_{L^q(\omega')} \| \nabla u \|_{L^p(\omega')}.
\end{equation*}
The left-hand side of the previous inequality is the sum of $2N$ terms $I_{11}+ I_{12} + \ldots + I_{N1} + I_{N2}$ where for every $k=1, \dots, N$,
\[
I_{k1}:= |h|^{-2} \IO \xi^2 (F_k(\tau_h \nabla u) - F_k(\nabla u) \cdot (\tau_h \nabla u - \nabla u),
\]
and
\[
I_{k2}:= |h|^{-2} \IO \xi^2 (F_k(\tau_h \nabla u) - F_k(\nabla u) \cdot \nabla \xi \xi (\tau_h u - u).
\]
Let $k=1, \dots, N$ fixed. We will find estimations on $I_{k1}$ and $I_{k2}.$
Due to \eqref{eq3}, $I_{k1}$ satisfies:
\[
I_{k1} \geq  \frac{4}{p^2} \| \xi |h|^{-1} (\tau_h \Hc_k - \Hc_k) \|_{L^2}^2.
\]
For $I_{k2}$, if $p > 2$, it follows from \eqref{eq4} and the H\"older inequality with exponents $2, p$ and $2p/(p-2)$ that
\bal
& |I_{k2}| \leq |h|^{-2} \IO | \xi \nabla \xi | | \tau_h u - u | |\tau_h \Hc_k - \Hc_k | \left ( |\tau_h \Hc_k |^{\frac{p-2}{p}} + |\Hc_k |^{\frac{p-2}{p}} \right ) \\
& \leq C \| |h|^{-1} (\tau_h u - u ) \|_{L^p(\omega_0)} \| \xi |h|^{-1} (\tau_h \Hc_k - \Hc_k) \|_{L^2} \left ( \int_{\omega_0} |\Hc_k|^2 + | \tau_h \Hc_k |^2 \right )^{\frac{p-2}{2p}} \\
& \leq C  \| \xi |h|^{-1} (\tau_h \Hc_k - \Hc_k) \|_{L^2},
\end{align*}
and if $p=2$, we simply use Cauchy-Schwarz inequality and we get :
\[
|I_{k2}| \leq C \| \xi |h|^{-1} (\tau_h \Hc_k - \Hc_k) \|_{L^2}.
\]
Bringing together all estimates, we then obtain
\[
\sum_{k=1}^N \left \| \xi \frac{\tau_h \Hc_k - \Hc_k}{h} \right \|_{L^2}^2 \leq C \left ( 1 + \sum_{k=1}^N \left \| \xi \frac{\tau_h \Hc_k - \Hc_k}{h} \right \|_{L^2} \right ).
\]
and we finally get
\[
\sum_{k=1}^N \left \| \frac{\tau_h \Hc_k - \Hc_k}{h} \right \|_{L^2(\omega)}^2 \leq C,
\]
for some constant $C$ that depends on $p, \|f\|_{W^{1,q}}, \|u\|_{W^{1,p}}$ and the distance between $\omega$ and $\partial \Omega$, but not on $h$. We have the desired result, that is, $\Hc_k \in W_{\loc}^{1,2}(\Omega)$, for $k=1, \dots, N$, and so $\Hc$ also.
\epr

If we consider the variational problem of Beckmann type
\beq \label{4.15}
\inf_{\sigma \in L^q(\Omega)} \left \{ \IO \inf_{\vr \in \Pc_x^{\sigma}} \sum_{k=1}^N c_k \left ( \frac{1}{q} \vr_k^q + \delta_k \vr_k \right ) : - \text{div } \sigma = f \right \},
\eeq
we then have the following Sobolev regularity result for the unique minimizer that generalizes Corollary $4.3$ in \cite{brasco2013congested}.
\begin{coro}
The solution $\sigma$ of \eqref{4.15} is in the Sobolev space $W_{\loc}^{1,r}(\Omega)$, where
$$
r = \left \{
\begin{aligned}
& 2 & \text{if } & p = 2, \\
& \text{any value } < 2, & \text{if } & p > 2 \text{ and } d = 2, \\
& \frac{dp}{dp-(d+p)+2}, & \text{if } & p > 2 \text{ and } d > 2.
\end{aligned} \right.
$$
\end{coro}

\bpr
By duality, we know the relation between $\sigma$ and any solution of the dual problem $u$
\[
\sigma = \sum_{k=1}^N (\nabla u \cdot v_k - \delta_k c_k)_+^{p-1} v_k.
\]
Since $u \in W^{1,q}(\Omega)$ is a weak solution of the Euler-Lagrange equation \eqref{4.10}, using \ref{thmreg} and \ref{lem4.2}, we have that the vector fields
\[
\Hc_k(x) =  (\nabla u (x) \cdot v_k - \delta_k c_k)_+^{\frac{p}{2}} v_k, \:  \: k=1, \dots, N,
\]
are in $W_{\loc}^{1,2}(\Omega)$. We then notice that $\sigma = \sum_{k=1}^N \sigma_k$ with
\[
\sigma_k = |\Hc_k|^{\frac{p-2}{p}} \Hc_k, \: \: k=1, \dots, N.
\]
The first case is trivial: we simply have $\sigma_k = \Hc_k \in W_{\loc}^{1,2}(\Omega)$. For the other cases, we use the Sobolev theorem. If $p>2$ and $d > 2$ then $\Hc_k \in L_{\loc}^{2*}(\Omega)$ with
\[
\frac{1}{2^*} = \frac{1}{2} - \frac{1}{N}.
\]
Applying \eqref{eq3} with $z= \tau_h \nabla u$ and $w= \nabla u$, we have
\[
\left | \frac{\tau_h \sigma_k - \sigma_k}{|h|} \right | \leq (p-1) \left ( |\tau_h \Hc_k |^{\frac{p-2}{p}} + |\Hc_k |^{\frac{p-2}{p}} \right ) \left | \frac{\tau_h \Hc_k - \Hc_k}{|h|} \right |.
\]
Since $ |\Hc_k |^{\frac{p-2}{p}} \in L_{\loc}^{\frac{2^* p}{p-2}}(\Omega)$, we have that the right-hand side term is in $L_{\loc}^r(\Omega)$ with $r$ given by
\[
\frac{1}{r} = \frac{p-2}{2^* p} + \frac{1}{2}.
\]
We can then control this integral
\[
\int \left | \frac{\tau_h \sigma_k - \sigma_k}{|h|} \right |^r dx.
\]
For the case $p>2$ and $d=2$, it follows from the same theorem that $\Hc_k \in L_{\loc}^s(\Omega)$ for every $s < +\infty$ and the same reasoning allows us to conclude.
\epr

This Sobolev regularity result can be extended to equations with weights such as
\begin{equation}
- \sum_{k=1}^N \text{div} \left ( b_k(x) (\nabla u(x) \cdot v_k - \delta_k c_k)_+^{p-1} v_k \right )  = f.
\end{equation}

An open problem is to investigate if one can generalize this Sobolev regularity result to the case where the $v_k$'s and $c_k$'s are in $C^1(\Ob)$.

\section{Numerical simulations}

\subsection{Description of the algorithm} \label{desalg}

We numerically approximate by finite elements solutions of the following minimization problem:
\beq \label{eq21}
\inf_{u \in W^{1,p}(\Omega)} J(u) := \mathbf{G}^*(\nabla u) - \langle f, u \rangle
\eeq
with $\mathbf{G}^*(\Phi) = \int_{\Omega} \mathcal{G}^*(x, \Phi(x)) \: dx$ for $\Phi \in L^p(\Omega)^d$ and $\langle f, w \rangle =  \int_{\Omega} u \: df$ for $w \in L^p(\Omega)$.
Let us recall that $\Omega$ is a bounded domain of $\RR^d$ with Lipschitz boundary and $f=f_+ - f_-$ is in the dual of $W^{1,p}(\Omega)$ with zero mean $\IO f = 0$.
We will use the augmented Lagrangian method described in \cite{benamou2013augmented} (that we will recall later). ALG2 is a particular case of the Douglas-Rachford splitting method for the sum of two nonlinear operators (see \cite{lions1979splitting} or more recently \cite{papadakis2014optimal}). ALG2 was used for transport problems for the first time in \cite{benamou2000computational}. Let a regular triangulation of $\Omega$ with typical meshsize $h$, let $E_h \subset W^{1,p}(\Omega)$  be the corresponding finite-dimensional space of $P_2$ finite elements of order $2$ whose generic elements are denoted $u_h$. Moreover, we approximate the terms $f$ by $f_h \in E_h$ (again with $\langle f_h, 1 \rangle = 0$) and $\mathbf{G}$ by a convex function $\mathbf{G}_h$. Let us consider the approximating problem
\beq \label{ppa}
\inf_{u_h \in E_h} J_h(u_h) := \mathbf{G}_h^*(\nabla u_h) - \langle f_h, u_h \rangle.
\eeq
and its dual
\beq \label{pda}
\sup_{\sigma_h \in F_h^d} \{ - \mathbf{G}_h(\sigma_h) : -\text{div}_h(\sigma_h) = f_h \}
\eeq
where $F_h$ is the space of $P_1$ finite elements of order $1$ and $-\text{div}_h(\sigma_h)$ may be understood as
\[
\langle \sigma_h, \nabla u_h \rangle_{F_h^d} = - \langle \text{div}_h (\sigma_h), u_h \rangle_{E_h}.
\]

\begin{theo}
If $u_h$ solves \eqref{ppa} then up to a subsequence, $u_h$ converges weakly in $W^{1,p}(\Omega)$ to a $u$ that solves \eqref{eq21} as $h \rightarrow 0$.
\end{theo}
It is a direct application of a general theorem (see \cite{benamou2013augmented} and \cite{gabay1976dual} for similar results and more details). Using the discretization by finite elements, \eqref{eq21} becomes
\beq \label{eq31}
\inf_{u \in \RR^n} J(u) := \Ff(u) + \Gf^*(\Lambda u)
\eeq
where $\Ff : \RR^n \rightarrow \RR \cup \{ + \infty \}, \Gf : \RR^m \rightarrow \RR \cup \{ + \infty \}$ are two convex l.s.c. and proper functions and $\Lambda$ is an $m \times n$ matrix with real entries. $\Lambda$ is the discrete analogue of $\nabla$. The dual of \eqref{eq31} then reads as
\beq \label{eq32}
\sup_{\sigma \in \RR^m} - \Ff^*(-\Lambda^T \sigma) - \Gf(\sigma)
\eeq

We say that a pair $(\bar{u}, \bar{\sigma}) \in \RR^n \times \RR^m$ satisfies the primal-dual extremality relations if:
\beq \label{eq33}
- \Lambda^T \bar{\sigma} \in \partial \Ff(\bar{u}), \bar{\sigma} \in \partial \Gf^*(\Lambda \bar{u}).
\eeq
It means that $\bar{u}$ solves \eqref{eq31} and that $\bar{\sigma}$ solves \eqref{eq32} and moreover, \eqref{eq31} and \eqref{eq32} have the same value (no duality gap). It is equivalent to find a saddle-point of the augmented Lagrangian function for $r>0$ (see \cite{fortin1983augmented, gabay1976dual} for example)
\beq \label{lagd}
L_r(u, q, \sigma) := \Ff(u) + \Gf^*(q) + \sigma \cdot (\Lambda u - q) + \frac{r}{2} | \Lambda u - q |^2, \: \forall (u,q,\sigma) \in \RR^n \times \RR^m \times \RR^m.
\eeq
It is the discrete formulation of the corresponding augmented Lagrangian function
\beq \label{lag}
L_r(u,q, \sigma) := \IO \mathcal{G}^*(x, q(x)) \: dx - \langle u, f \rangle + \langle \sigma, \nabla u - q \rangle + \frac{r}{2} |\nabla u(x) - q(x) |^2
\eeq
and the variational problem of \eqref{eq31} is
\begin{equation}
\label{eq40}
\inf_{u,q} \left \{ \int_{\Omega} \mathcal{G}^*(x, q(x)) \: dx - \int_{\Omega} u(x) f(x) \: dx  \right \}.
\end{equation}

subject to the constraint that $\nabla u = q$.

The augmented Lagrangian algorithm ALG2 involves building a sequence $(u^k, q^k, \sigma^k) \in \RR \times \RR^d \times \RR^d$ from initial data $(u^0, q^0, \sigma^0)$ as follows:
\begin{enumerate}
\item Minimization problem with respect to $u$:
\[
u^{k+1} := \text{argmin}_{u \in \RR^n} \left \{ \Ff(u) + \sigma^k \cdot \Lambda u + \frac{r}{2} |\nabla u - q^k |^2) \right \}
\]

That is equivalent to solve the variational formulation of Laplace equation
\[
-r(\nabla u^{k+1} - \text{div}(q^k)) = f + \text{div}(\sigma^k) \text{ in } \Omega
\]
with the Neumann boundary condition
\[
r \frac{\partial u^{k+1}}{\partial \nu} = r q^k \cdot \nu - \sigma^k \cdot \nu \text{ on } \partial \Omega.
\]

This is where we use the Galerkin discretization by finite elements.

\item Minimization problem with respect to $q$:
 \[
 q^{k+1} := \text{argmin}_{q \in \RR^d} \left \{ \Gf^*(q) - \sigma^k \cdot q + \frac{r}{2} |\nabla u^{k+1} - q |^2) \right \}
 \]

 \item Using the gradient ascent formula for $\sigma$
 \[
 \sigma^{k+1} = \sigma^k + r(\nabla u^{k+1} - q^{k+1}).
 \]

\begin{theo}
Given $r>0$. If there exists a solution to the primal-dual extremality relations \eqref{eq33} and $\Lambda$ has full column-rank then there exists an $(\bar{u}, \bar{\sigma}) \in \RR^n \times \RR^m$ satisfying \eqref{eq33} such that the sequence $(u^k, q^k, \sigma^k)$ generated by the ALG2-scheme above satisfies
\beq \label{eq39}
u^k \rightarrow \bar{u}, q^k \rightarrow \Lambda \bar{u}, \sigma^k \rightarrow \bar{\sigma} \text{ as } k \rightarrow +\infty.
\eeq
\end{theo}

\end{enumerate}

We directly apply a general theorem whose proof can be found in \cite{eckstein1992douglas} (Theorem $8$), following contributions of \cite{fortin1983augmented, gabay1976dual, lions1979splitting} to the analysis of splitting methods.

\subsection{Numerical schemes and convergence study}

We use the software FreeFem++ (see \cite{MR3043640}) to implement the numerical scheme. We take the Lagrangian finite elements and notations used in \ref{desalg}, $P_2$ FE for $u_h$ and $P_1$ FE for $(q_h, \sigma_h)$. $\Lambda u_h$ is the projection on $P_1$ of the operator $\Lambda$, that is, $\nabla u_h$. The first step and the third one are always the same and only the second one varies with our different test cases. We indicate the numerical convergence of ALG2 iterations by the $\cdot^k$ superscript and the convergence of finite elements discretization by the $\cdot_h$ subscript. For our numerical simulations, we work with the space dimension $d=2$ and we choose for $\Omega$ a $2D$ square $(x=(x_1, x_2) \in [0,1]^2)$. We make tests with different $f$ :
\[
f_-^1 := e^{-40*((x_1-0.75)^2+(x_2-0.25)^2} \text{ and } f_+^1 := e^{-40*((x_1-0.25)^2+(x_2-0.65)^2)},
\]

\[
f_-^2 := e^{-40*((x_1-0.5)^2+(x_2-0.15)^2)} \text{ and } f_+^2 := e^{-40*((x_1-0.5)^2+(x_2-0.75)^2)},
\]
In the third case, we take $f_-^3$ a constant density and $f_+^3$ is the sum of three concentrated Gaussians
\begin{multline*}
f_+^3(x_1, x_2) = e^{-400*((x-0.25)^2+(y-0.75)^2)} +e^{-400*((x-0.35)^2+(y-0.15)^2)} \\
+ e^{-400*((x-0.85)^2+(y-0.7)^2)}.
\end{multline*}

We also make tests with non-constant $c_k$ :
\[
g^1(x_1, x_2) = 3 - 2*e^{-10*((x_1-0.5)^2+(y_2-0.5)^2)}.
\]

As specified above, we use a triangulation of the unit square with $n=1/h$ element on each side. We use the following convergence criteria:
\begin{enumerate}
\item DIV.Error $ = \left ( \int_{\Omega_h} (\text{div} \sigma_h^k + f )^2 \right ) ^{1/2}$ is the $L^2$ error on the divergence constraint.

\item BND.Error $= \left ( \int_{\partial \Omega_h} (\sigma_h^k \cdot \nu)^2 \right)^{1/2}$ is the $L^2 (\partial \Omega_h)$ error on the Neumann boundary condition.

\item DUAL.Error $ = \max_{x_j} | \Gc(x_j, \sigma_h^k(x_j)) + \Gc^*(x_j, \nabla u_h^k(x_j)) - \nabla u_h^k(x_j) \cdot \sigma_h^k(x_j) |$ where the maximum is with respect to the vertices $x_j$.
\end{enumerate}
The first two criteria represent the optimality conditions for the minimization of the Lagrangian with respect to $u$ and the third one is for maximization with respect to $\sigma$.

We make tests for two models. In the first one, the directions are the same as in the cartesian model and the volume coefficients are not necessarily constant. In the second one, the directions are the same than in the hexagonal one and the volume coefficients are equal to $1$ (it is simpler to compute $\Gc(x, \sigma)$). That is, $v_k = \exp(ik \pi/3)$ and $\delta_k c_k = 1$ for $k=1,\dots,6$. We call these models still the cartesian one, the hexagonal one respectively. The cartesian one is much easier since we can separate variables. $\Gf = \Gf_1 + \Gf_2$ with $\Gf_i (x,q) = \frac{b_i}{p} (|q_i| - \delta_i c_i(x))^p_+$ so that the second step of ALG2 is equivalent to solve the pointwise problem
\[
\inf_{q} \frac{1}{p} (|q| - c(x))^p_+ + \frac{r}{2} |q - \tilde{q}^k|^2
\]
where $\tilde{q}^k = \nabla u^{k+1} + \frac{\sigma^k}{r}$. This amounts to set $q^{k+1} = \lambda \tilde{q}^k$ and to solve this equation in $\lambda$
\[
(\lambda |\tilde{q}^k| - c(x))^{p-1}_+ + r \lambda |\tilde{q}^k| = r |\tilde{q}^k| = 0
\]
with $\lambda \geq 0$.
 We can use the dichotomy algorithm.

 For the hexagonal one, we use Newton's method. Since the function of which we seek the minimizer has its Hessian matrix that is definite positive, we can use the inverse of this Hessian matrix.

 We show the results of numerical simulations after $200$ iterations for both models.

\begin{table}
\begin{tabular}{|c|c|c|c|c|}
\hline Test case & DIV.Error & BND.Error & DUAL.Error & Time execution (seconds) \\
\hline 1 & 8.4745e-05 & 0 & 3.6126e-06 & 436 \\
\hline 2 & 2.2536e-05 & 8.8705e-04 & 3.0663e-05 & 4764 \\
\hline 3 & 5.2141e-05 & 1.4736e-04 & 1.1556e-02 & 792 \\
\hline 4 & 1.1823e-05 & 7.6776e-04 & 8.7412e-06 & 170 \\
\hline 5 & 1.1629e-05 & 0 & 9.7498e-04 & 285 \\
\hline 6 & 3.1544e-04 & 1.0958 & 7.8350e-07 & 445 \\
\hline 7 & 4.1373e-04 & 1.1710 & 4.8113e-04 & 4657 \\
\hline
\end{tabular}
\caption{Convergence of the finite element discretization for all test cases.}
\end{table}

We notice that length of arrows are proportional to transport density. Level curves correspond to the density term of the source/sink data to be transported. In \ref{casf2}, the case $p=1.01$ means that there is much congestion. The case $p=2$ is reasonable congestion and in the last one $p=100$, there is little congestion.  When there are obstacles, the criteria BND.Error is not very good. Indeed, the flow comes right on the obstacle and it turns fast. In the other side of the obstacle, the flow is tangent to the border. Many other cases may of course be examined (other boundary conditions, obstacles, coefficients depending on $x$, different exponents $p$ for the different components of the flow...).

\begin{figure}
\includegraphics[scale=0.461]{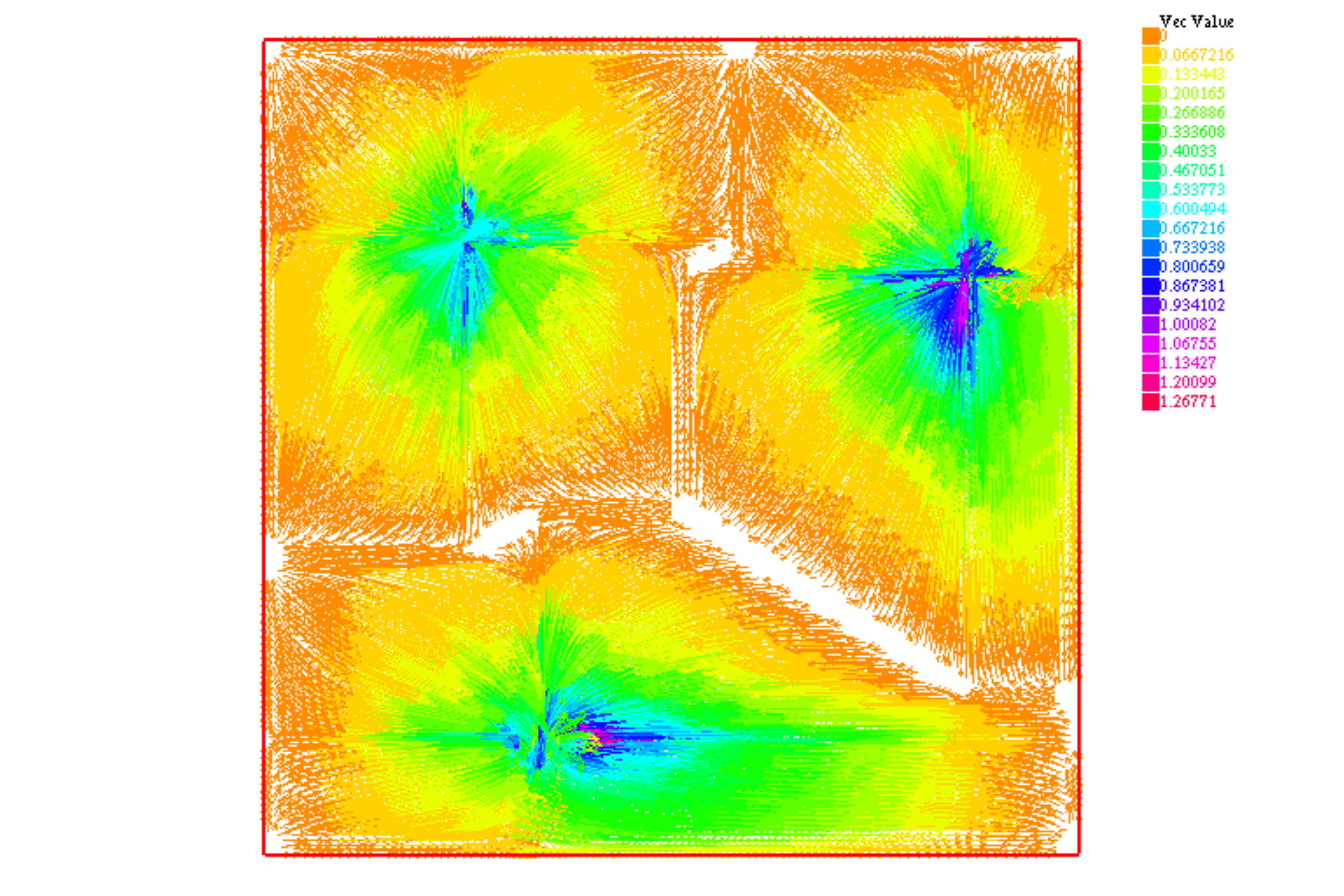}
\caption{Test case 1 : cartesian case ($d=2$) with $f=f^3$, $c_k$ constant and $p=10$.}
\end{figure}

\begin{figure}[htbp]
\hspace*{-1.3cm}
\centering
\includegraphics[scale=0.461]{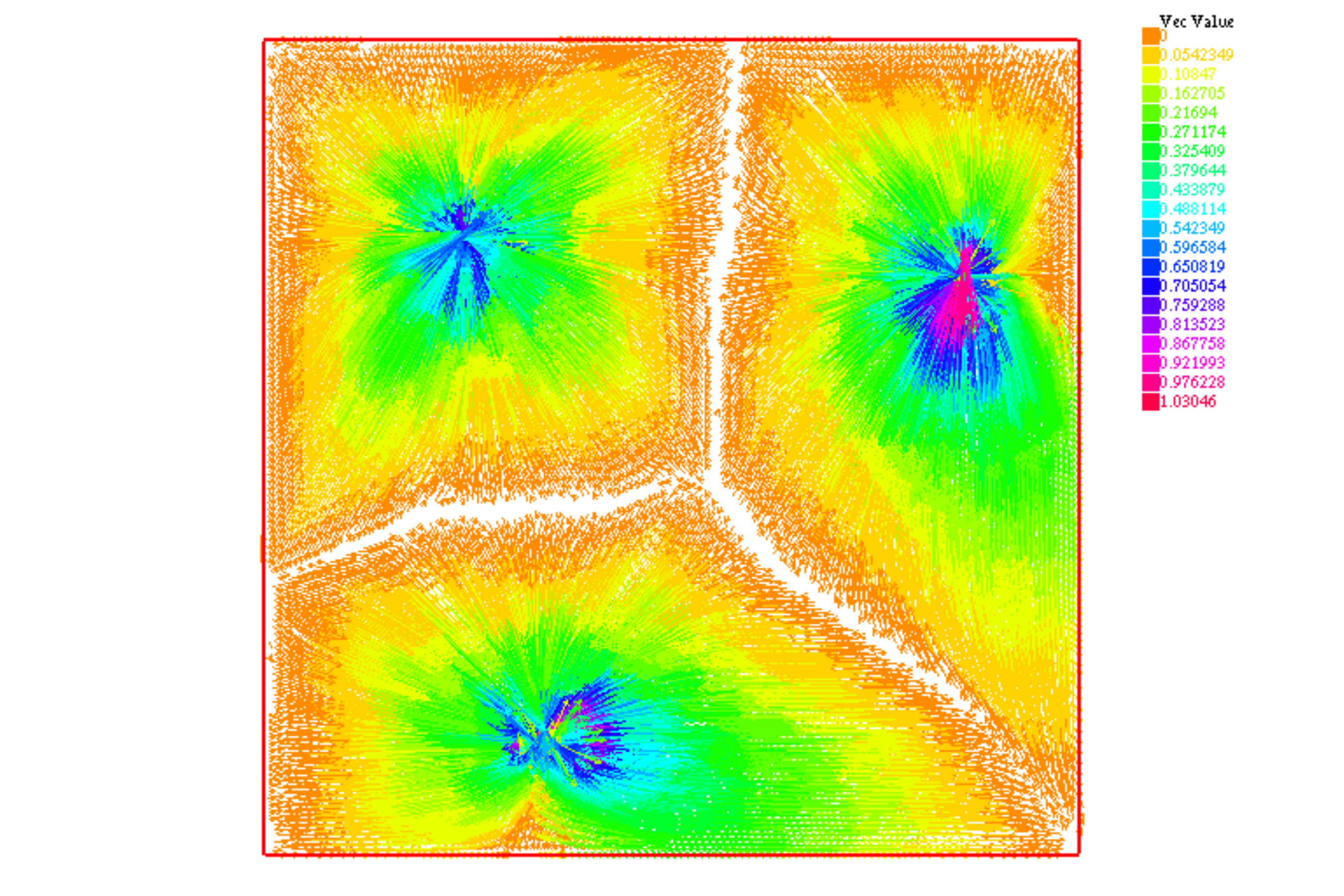}
\caption{Test case 2 : hexagonal case ($d=2$) with $f=f^3$, $c_k$ constant and $p=3$.}
\end{figure}

\begin{figure}
\includegraphics[scale=0.35]{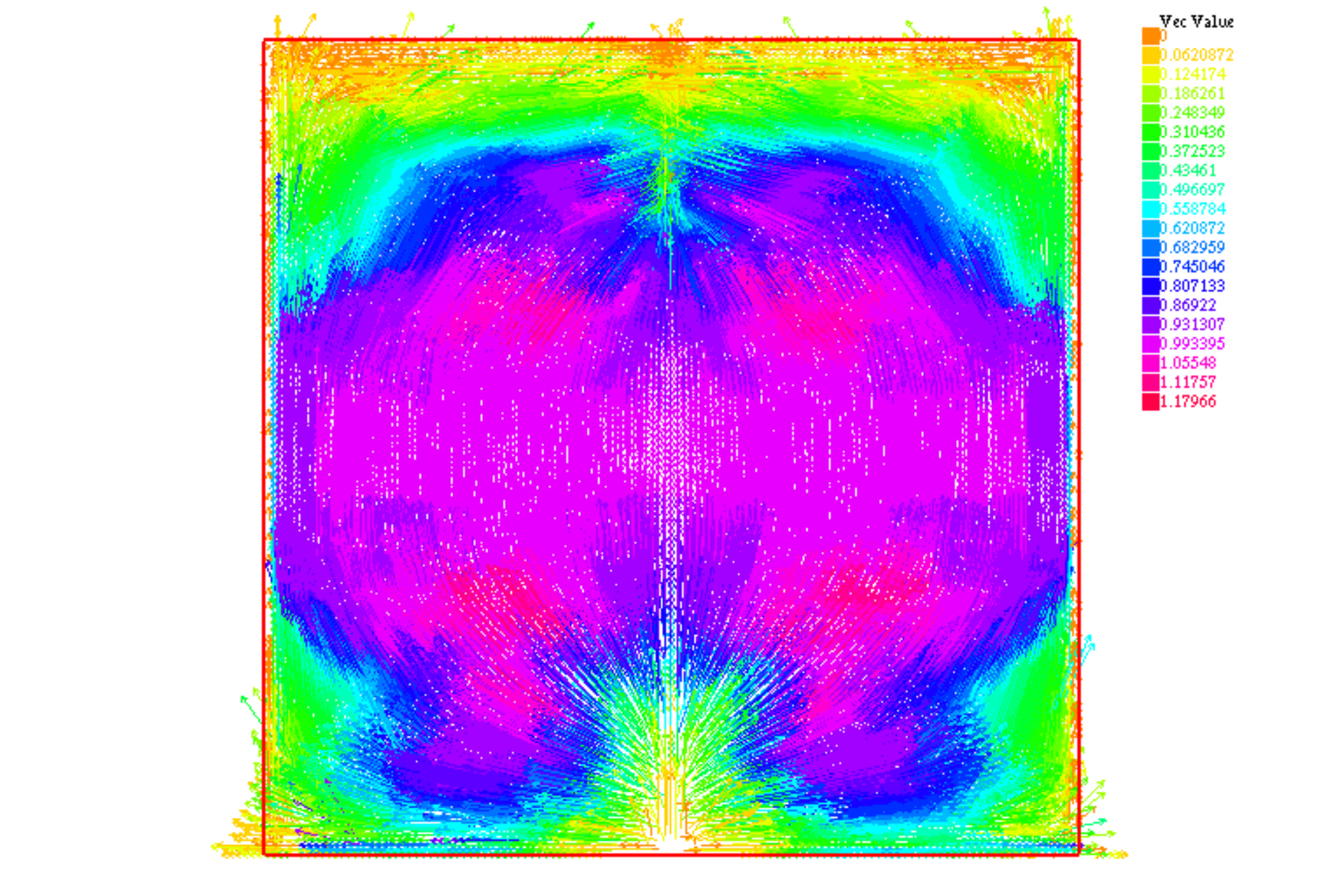}
\includegraphics[scale=0.35]{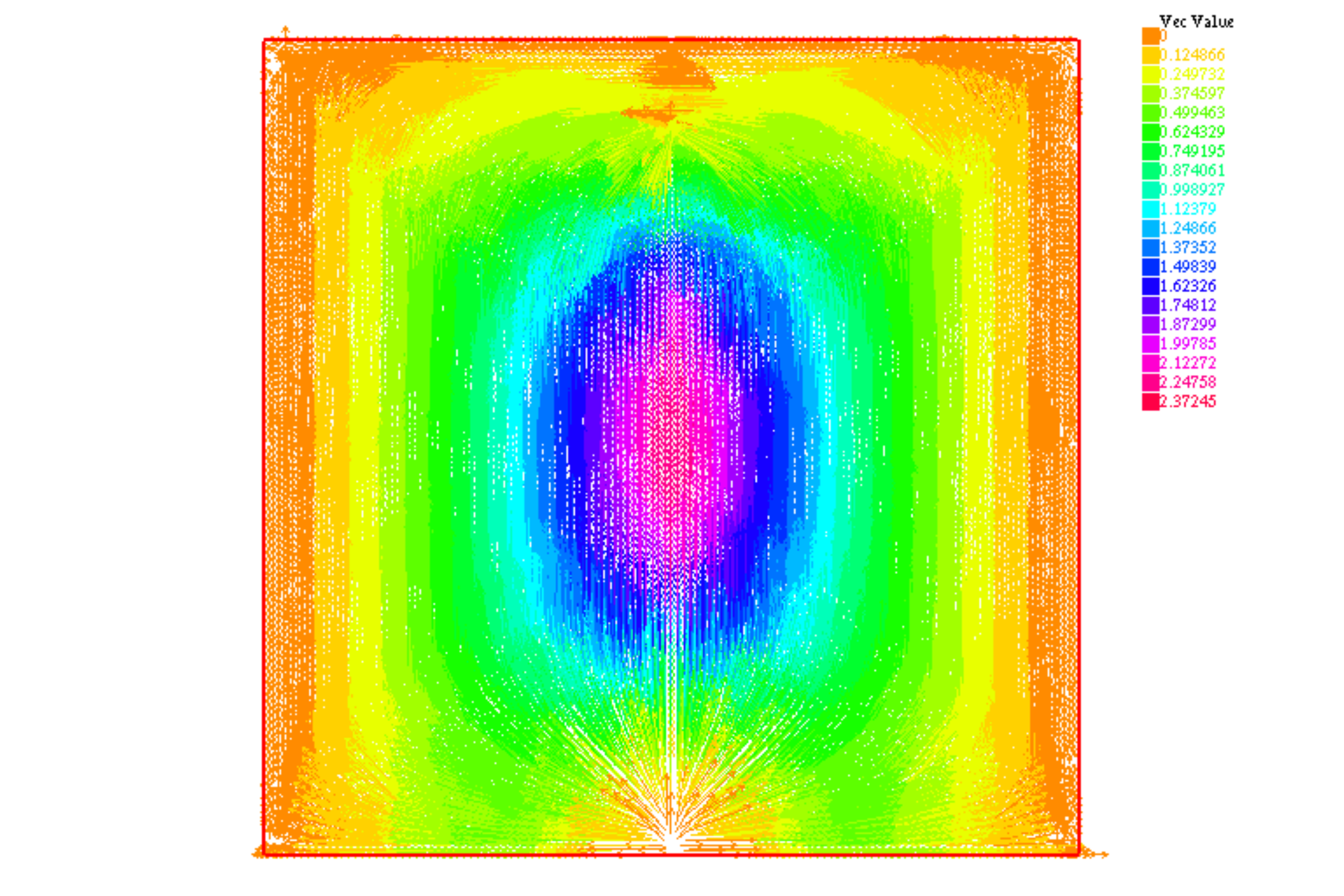}
\includegraphics[scale=0.35]{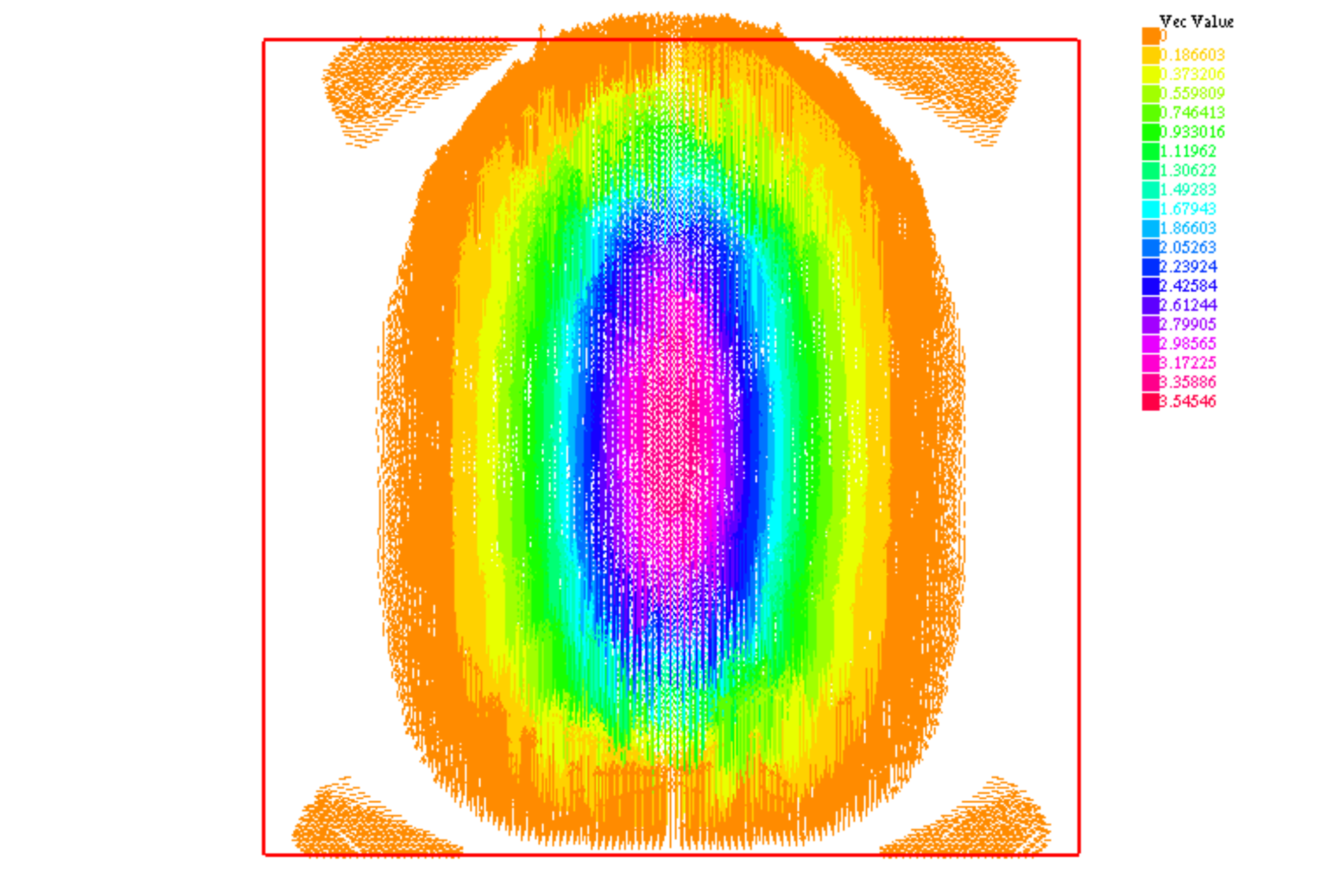}
\caption{Test cases 3, 4 and 5: cartesian case ($d=2$) with $f=f^2$, $c_k$ constant and $p=1.01, 2, 100$.}
\label{casf2}
\end{figure}

\begin{figure}
\includegraphics[scale=0.461]{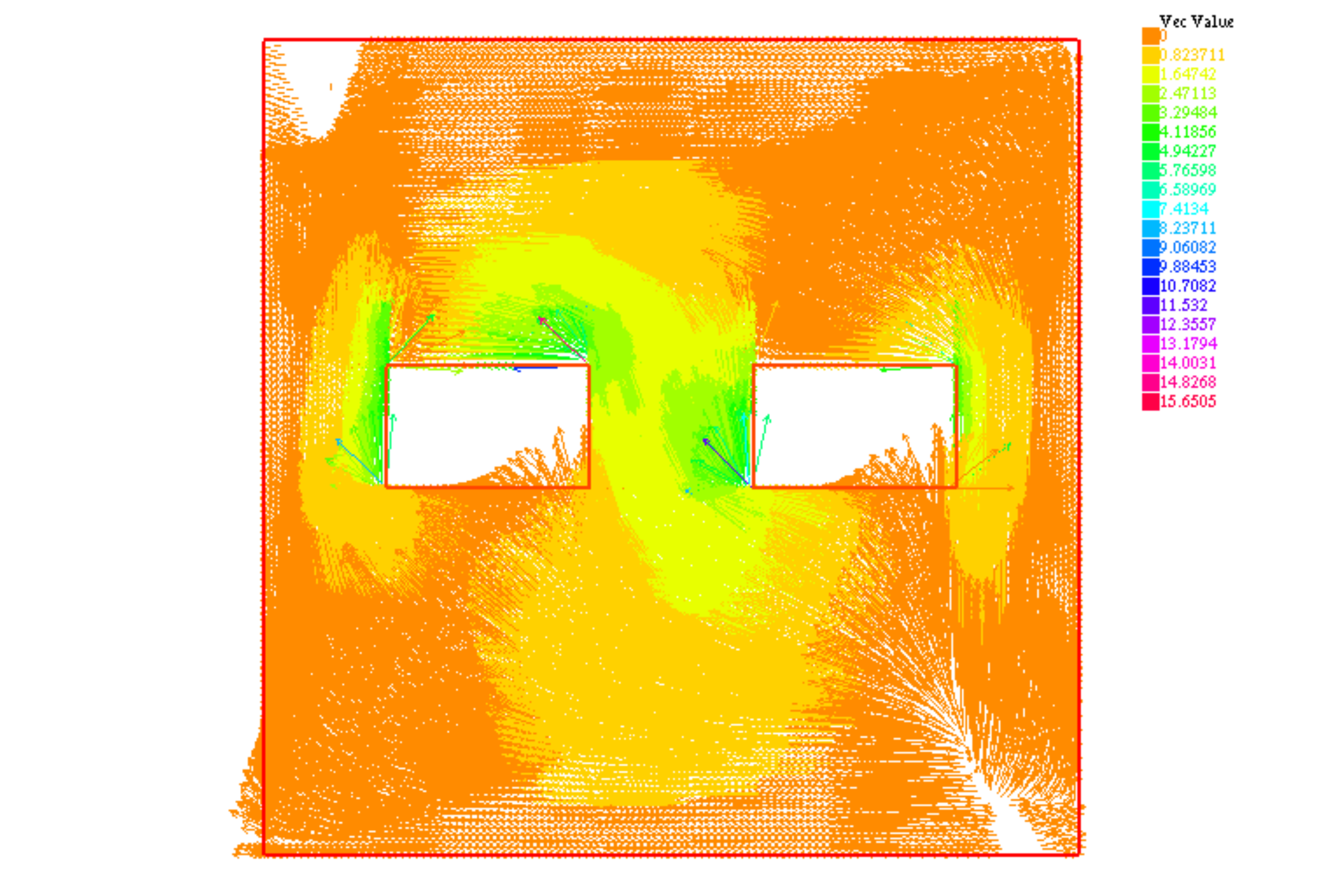}
\caption{Test case 6 : cartesian case ($d=2$) with $f=f^1$, $c_k=g^2$, $p=3$ and two obstacles.}
\end{figure}

\begin{figure}[htbp]
\includegraphics[scale=0.461]{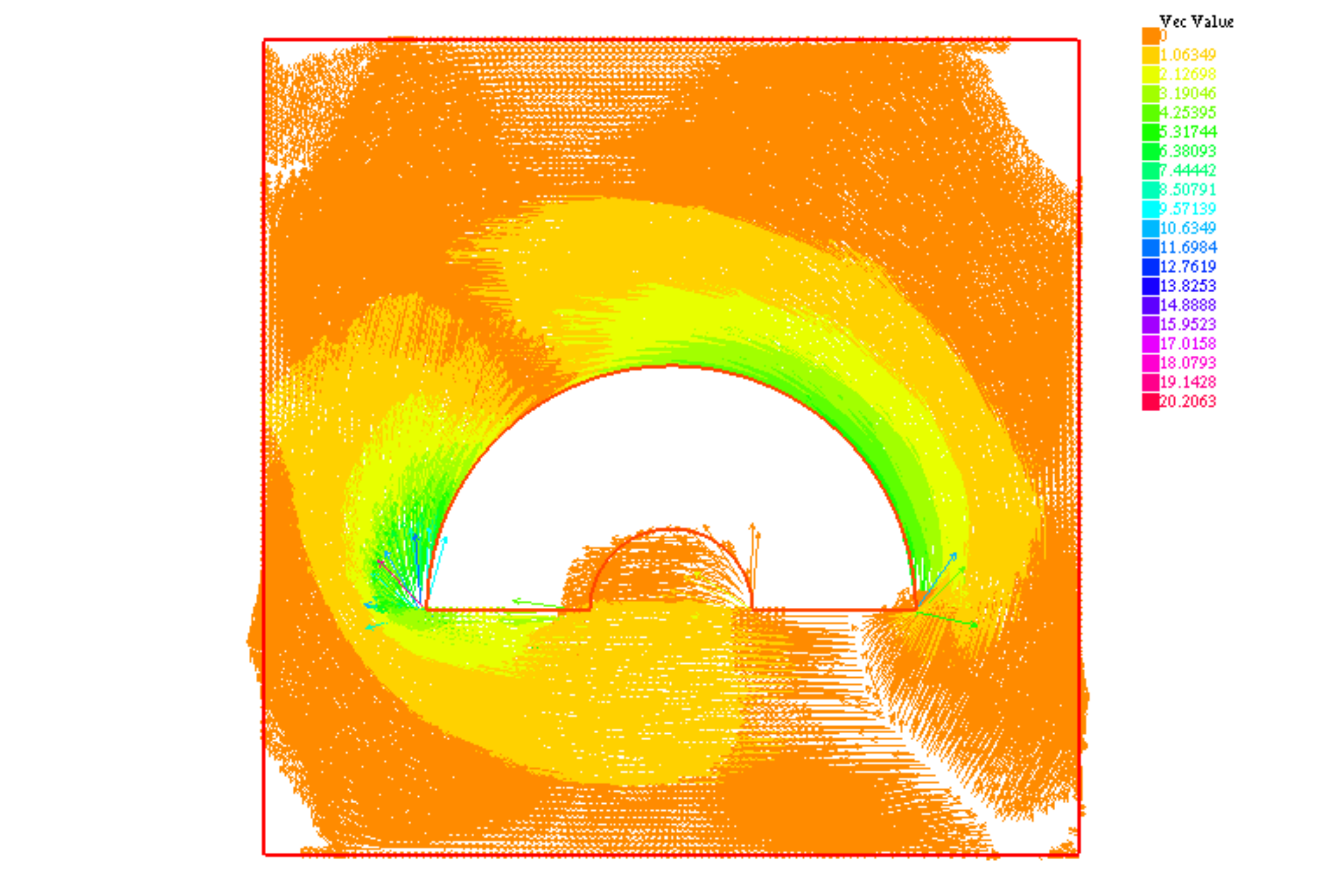}
\caption{Test case 7 : hexagonal case ($d=2$) with $f=f^1$, $c_k$ constant, $p=3$ and an obstacle.}
\end{figure}

\textbf{Acknowledgements} The author would like to thank Guillaume Carlier for his extensive help and advice as well as Jean-David Benamou and Ahmed-Amine Homman for their explanations on FreeFem ++.

\bibliographystyle{plain}
\bibliography{generalref}

\end{document}